\documentclass[11pt]{amsart}  
\usepackage{amssymb,amsmath,amsthm,amscd}
\usepackage{graphicx}
\usepackage{lscape}
\newcommand{\ep}[0]{\epsilon}

\newcommand{\ov}[1]{\frac{1}{#1}}
\newcommand{\f}[1]{\mathbb{#1}}

\newcommand{\ar}[1]{\mathbb{R}^{#1}}

\newcommand{\pth}[0]{{\partial_\theta}}

\numberwithin{equation}{section}

\newtheoremstyle{fancy1}{10pt}{10pt}{\itshape}{12pt}{\textsc\bgroup}{.\egroup}{8pt}{
}
\newtheoremstyle{fancy2}{10pt}{10pt}{}{12pt}{\itshape}{.}{8pt}{ }

\theoremstyle{fancy1}

\newtheorem{lem}[equation]{Lemma}

\newtheorem{thm}[equation]{Theorem}
\newtheorem*{thm*}{Theorem}

\newtheorem{main}{Theorem}
\newtheorem*{main*}{Theorem}

\newtheorem*{cor*}{Corollary}
\newtheorem*{prop*}{Proposition}

\newtheorem*{problem*}{Problem}

\theoremstyle{fancy2}

\newtheorem{rem}[equation]{Remark}
\newtheorem*{rems*}{Remarks}

\newtheorem*{rem*}{Remark}

\newtheorem*{example*}{Example}

\newcommand{\cref}[1]{Corollary~\ref{#1}}

\newcommand{\lref}[1]{Lemma~\ref{#1}}

\newcommand{\rref}[1]{Remark~\ref{#1}}
\newcommand{\tref}[1]{Theorem~\ref{#1}}

\def\con#1=#2(#3){#1 \equiv #2 \bmod{#3}}

\begin{document}

\title{On the Moduli Spaces of Metrics with Nonnegative Sectional Curvature}

\author{McFeely Jackson Goodman}
\address{University of Pennsylvania}
\email{mcfeelyg@math.upenn.edu}

\begin{abstract}
The Kreck-Stolz \(s\) invariant is used to distinguish connected components of the moduli space of positive scalar curvature metrics.  We use a formula of Kreck and Stolz to calculate the \(s\) invariant for metrics on \(S^n\) bundles with nonnegative sectional curvature.  We then apply it to show that the moduli spaces of metrics with nonnegative sectional curvature on certain 7-manifolds have infinitely many path components.  These include the first non-homogeneous examples of this type and certain positively curved Eschenburg and Aloff-Wallach spaces.     
\end{abstract}

\maketitle 

The description of manifolds with positive and nonnegative curvature is an important topic in Riemannian geometry.  On  manifolds \(M^n\) which support such metrics, the next question is to describe the space of all such metrics.  The proper space to study is the moduli space, or the quotient of the space of positive or nonnegative curvature metrics by the pullback action of the diffeomorphism group, which we denote by \(\mathfrak{M}\).   The number of connected components of the moduli space serves as a coarse quantification of nonnegative curvature metrics.

Kreck and Stolz \cite{KS} defined the \(s\) invariant of the path components of the moduli space of positive scalar curvature metrics \(\mathfrak{M}_{\text{scal}>0}(M)\) on certain spin manifolds.  If two metrics on \(M\) yield different values of \(s,\) they cannot be connected by a path maintaining positive scalar curvature (even up to diffeomorphism).   Kreck and Stolz used the invariant to show that for a (4n+3)-manifold with a unique spin structure the space of such metrics is either empty or has infinitely many components.  

The \(s\) invariant can also be used to distinguish path components of the moduli spaces of metrics with stronger curvature conditions.  In \cite{W}, Wraith showed that \(\mathfrak{M}_{\text{Ric}>0}\) has infinitely many components for every \(4n+3\geq7\) dimensional homotopy sphere which bounds a parallelisable manifold.   In \cite{KS} Kreck and Stolz studied the \(s\) invariant for two families of 7-manifolds.  The first are the total spaces \(N_{k,l}^7\) of principal \(S^1\) bundles over \(\f{C}P^2\times\f{C}P^1\).  These spaces are also described as the homogeneous spaces \(S^5\times S^3/S_{k,l}^1\), see \cite{WZ}.  Using the diffeomorphism invariants of \cite{KS0} they show that each \(N^7_{k,l}\), with \(k\) even and gcd\((k,l)=1\), is diffeomorphic to infinitely many manifolds in the same family.  Calculating the \(s\) invariants for the Einstein metrics described in \cite{WZ} and the homogeneous metrics induced from \(S^5\times S^3\) it follows that \(\mathfrak{M}_{\text{Ric}>0}(N^7_{k,l})\) and \(\mathfrak{M}_{\text{sec}\geq0}(N^7_{k,l})\) have infinitely many path components.  
  
The second family are the Aloff-Wallach spaces \(W_{k,l}^7=SU(3)/S^1_{k,l}\).  Using the diffeomorphism classification in \cite{KS1} they show that some \(W_{k,l}^7\) are diffeomorphic to finitely many other manifolds in the family.  As these spaces have homogeneous metrics of positive curvature when \(kl(k+l)\neq0,\) they exhibit some examples where \(\mathfrak{M}_{\text{sec}>0}(W^7_{k,l})\) has more than one component.  We note though that for all known families of manifolds admitting positive sectional curvature metrics, only finite subfamilies have the same cohomology ring.

 The same methods are used in \cite{DKT} to calculate the \(s\) invariants of homogeneous metrics on the total spaces of \(S^1\) bundles \(N^{4n+3}_{k,l}\) over \(\f{C}P^{2n}\times\f{C}P^1\).  As observed in \cite{WZ}, for fixed \(n\) and \(l\) this family contains only finitely many diffeomorphism types.  It follows that for any fixed \(n\) and \(l\) there exists some \(k_0\) such that \(\mathfrak{M}_{\text{sec}\geq0}(N^{4n+3}_{k_0,l})\) has infinitely many path components.  We note however that for \(n\geq2\) and \(|l|>2\) this does not identify any specific manifold having that property.  It is also of note that all previous results concerning \(\mathfrak{M}_{\text{sec}>0}\) and \(\mathfrak{M}_{\text{sec}\geq0}\) calculate the \(s\) invariant only for homogeneous metrics admitting an \(S^1\) action with  geodesic orbits.  

We identify further 7-manifolds with \(\mathfrak{M}_{\text{sec}\geq0}\) and \(\mathfrak{M}_{\text{Ric}>0}\) having infinitely many path components.  As in previous examples, the manifolds are total spaces of \(n-\)sphere bundles.  In our case, however, the metrics are not homogeneous and the \(S^n\) fibers are not totally geodesic.  

\smallskip
The first set of examples are total spaces \(M_{m,n}\) of \(S^3\) bundles over \(S^4\).  Such bundles are classified by pairs of integers \((m,n)\in\pi_3(SO(4))\cong\f{Z}\oplus\f{Z}\).  The second set are total spaces \(S_{a,b}\) of \(S^3\) bundles over \(\f{C}P^2\) which are not spin.  They are classified by two integers \(a,b\) describing the first Pontryagin class and the Euler class.  Grove and Ziller \cite{GZ,GZ2} showed that \(M_{m,n}\) and \(S_{a,b}\) admit metrics of nonnegative sectional curvature.
\begin{main}\label{thm}
Let \(m,n,a,b\in\f{Z}\) with \(n\neq0\) and \(a\neq b\).  Then for \(M=M_{m,n}\) or \(S_{a,b}\) the moduli spaces \(\mathfrak{M}_{\text{\normalfont{sec}}\geq0}(M)\) and \(\mathfrak{M}_{\text{\normalfont{Ric}}>0}(M)\) have infinitely many path components.
\end{main}
Note that the family \(M_{m,\pm1}\) includes \(S^7\) and the exotic Milnor spheres.  By \cite{EZ} Proposition 6.7 the manifold \(S_{-1,a(a-1)}\) is diffeomorphic to the Aloff-Wallach space \(W^7_{a,1-a}\) discussed above.  But in general \(M_{m,n}\) and \(S_{a,b}\) do not have the homotopy type of a 7-dimensional homogeneous space, e.g. when \(|H^4(M_{m,n},\f{Z})|=|n|\notin \{1,2,10\}\) or  \(|H^4(S_{a,b},\f{Z})|=|a-b|=2\) mod \(3\) respectively.  

\smallskip
To describe the final set of manifolds, we start with \(S^2\) bundles \(\bar{N}_t\) and \({N}_t\) over \(\f{C}P^2\) which are spin, respectively not spin, and are classified by an integer \(t\) describing the first Pontryagin class.   The 7-manifolds \(\bar{M}_{a,b}^t\) and \({M}_{a,b}^t\) are the total spaces of \(S^1\) bundles over \(\bar{N}_t\) and \({N}_t\) respectively, classified by two additional integers \(a\) and \(b\), with gcd\((a,b)=1\), describing the Euler class.  Escher and Ziller \cite{EZ} showed that  \(M^t_{a,b}\) and \(\bar{M}^{2t}_{a,b}\) admit metrics of non-negative sectional curvature such that \(S^1\) acts by isometries.

\begin{main}\label{cthm}
(a)\indent Let \(a,b,t\in\f{Z}\) with gcd\((a,b)=1\) and \(t(a+b)^2\neq ab\). Then \(\mathfrak{M}_{\text{\normalfont{sec}}\geq0}(M^{t}_{a,b})\) and \(\mathfrak{M}_{\text{\normalfont{Ric}}>0}(M^{t}_{a,b})\) have infinitely many  path components. 

(b)\indent Let \(a,b,t\in\f{Z}\) with gcd\((a,2b)=1\) . Then \(\mathfrak{{M}}_{\text{\normalfont{sec}}\geq0}(\bar{M}^{2t}_{a,2b})\) and \(\mathfrak{{M}}_{\text{\normalfont{Ric}}>0}(\bar{M}^{2t}_{a,2b})\) have infinitely many  path components.

\end{main}
In \cite{EZ} Corollary 6.4 it was shown  that the manifold \(M^{-1}_{a,b}\) is the Eschenburg biquotient \(F_{a,b}=S^1_{a,b,a+b}\backslash SU(3)/S^1_{0,0,2a+2b}\).  These are the only Eschenburg biquotients admitting free \(S^1\) actions and when \(ab>0\) they admit metrics of positive sectional curvature, see \cite{Es}.  Furthermore \(M^1_{a,b}\) is the Aloff-Walach space \(W_{a,b}\).  We have thus as an immediate corollary \begin{cor*}
 For \(M=W_{a,b}\) or \(M=F_{a,b}\) the moduli spaces \(\mathfrak{M}_{\text{\normalfont{sec}}\geq0}(M)\) and \(\mathfrak{M}_{\text{\normalfont{Ric}}>0}(M)\) have infinitely many  path components. \end{cor*}   

 These are the first examples where \(\mathfrak{M}_{\text{sec}\geq0}(M)\)  has infinitely many components, some of which contain metrics with positive sectional curvature.  We note that in \cite{EZ} one finds further examples of positively curved Eschenburg spaces which are diffeomorphic to some of the manifolds \(S_{a,b}\) or \(M^t_{a,b}\) and so the same conclusion holds.    
 
 By Corollary 7.8 of \cite{EZ}  \(\bar{M}^{0}_{a,2b}\) is diffeomorphic to the homogeneous space \(N^7_{2b,a}\), and hence \tref{cthm} part (b) generalizes the Kreck-Stolz examples.  But again in general \(\bar{M}^{2t}_{a,2b}\) and \(M_{a,b}^t\) do not have the homotopy type of a 7-dimensional homogeneous space, e.g. when \(|H^4(\bar{M}^{2t}_{a,2b},\f{Z})|=|a^2-8tb^2|=2\) mod \(3\) or \(|H^4(M^{t}_{a,b},\f{Z})|=|t(a+b)^2-ab|=2\) mod \(3\).   
  
  \smallskip
  The strategy of the proof is as follows.  We calculate the \(s\) invariant with topological data on the associated disc bundle of the sphere bundle.   In \tref{propB} we extend the metric with sec\(\ \geq0\) on each sphere bundle to a metric of positive scalar curvature on the associated disc bundle which is a product near the boundary.   If the disc bundle is a spin manifold Kreck and Stolz \cite{KS} obtained a formula for the \(s\) invariant in terms of the index of the Dirac operator, which vanishes since the scalar curvature is positive, and topological data on a bounding manifold, see \tref{ksinv}.  \tref{thm} follows easily:  the manifolds \(M_{m,n}\) and \(S_{a,b}\) are classified up to diffeomorphism in \cite{CE} and \cite{EZ} respectively.  In particular each sphere bundle is diffeomorphic to infinitely many others.  Their computations easily yield the formula for the \(s\) invariant as well.  \tref{thm} follows since \(s\) is a polynomial in the integers \(a,b,m\) and \(n,\) where \(m,n\) satisfy \(ma+nb=1\).
  
  \tref{cthm} is more involved.  For part (b) we observe that \(\bar{M}^{t}_{a,b}\) is a spin manifold if and only if \(b\) is even, and in this case the associated disc bundle is a spin manifold as well.  The proof then proceeds as before although the proof that the metrics have positive scalar curvature is more involved.  We use the Kreck Stolz invariants \(s_1,s_2,s_3\in\f{Q}/\f{Z}\) of \cite{KS1} to obtain infinitely many circle bundles diffeomorphic to each manifold.  For part (a)  the manifolds \(M^t_{a,b}\) are always spin but the disc bundles are not.  Here we use another formula from \cite{KS}, see \tref{ns}, which does not require knowledge of a spin bounding manifold, but requires that the bundle be a circle bundle and that the fibers be geodesics.  The latter condition does not hold for the metrics with sec\(\ \geq0\), so we first deform the metrics, preserving positive scalar curvature, until the fibers are geodesics and such that \(S^1\) still acts by isometries.  Then the strategy proceeds in the same way.
  
  We note that \(\bar{M}^{2t}_{a,2b+1}\) and the spin \(S^3\) bundles over \(\f{C}P^2\) also admit metrics with sec \(\geq0\), but they are not spin manifolds so the methods do not apply . The conditions \(a\neq b\) and \(n\neq0\) for \(S_{a,b}\) and \(M_{n,m}\) as well as \(t(a+b)^2\neq ab\) for \(M^t_{a,b}\) are required to ensure the manifolds have the correct cohomology ring for the diffeomorphism classifications.                         
  
\smallskip
 We point out that the case of \(S^3\) bundles over \(S^4\) was obtained independently by A. Dessai in \cite{De}.      

\smallskip  
I would like to thank my Ph. D. advisor Wolfgang Ziller for endless ideas and support.  

\section{Preliminaries}
\bigskip
Let \((M,g)\) be a \(4k-1\) dimensional Riemannian spin manifold with vanishing rational Pontryagin classes.  The Kreck-Stolz \(s\) invariant is defined in \cite{KS} intrinsically for a positive scalar curvature metric \(g\) on \(M\) :

\[s(M,g)=-\eta(D(M,g))-a_k\eta(B(M,g))-\int_Md^{-1}(\hat{A}+a_kL)(p_i(M,g)).\]

\noindent Here \(D\) is the Dirac operator on the spinor bundle, \(B\) is the signature operator on  differential forms and  \(\eta\) is the spectral asymmetry invariant of a differential operator defined in [APS].  \(p_i(M,g)\) are the Pontryagin forms defined in terms of the curvature tensor of \(g\). \(\hat{A}\) and \(L\) are the Hirzebruch polynomials and \(a_k=(2^{2k+1}(2^{2k-1}-1))^{-1}\). \(d^{-1}\) represents a form whose exterior derivative is the indicated form.  Kreck and Stolz \cite{KS} showed that existence of this form and uniqueness of the integral follow from the vanishing of the rational Pontryagin classes.  They further showed the invariant depends only on the connected component of \(g\) in \(\mathfrak{M}_{\text{scal}>0}\).  

 Kreck and Stolz use the Atiyah-Patodi-Singer index theorem \cite{APS} and the Lichnerowitz theorem for manifolds with boundary (\cite{APS2}, p.416) to prove
 
 \begin{thm}\label{ksinv}\cite{KS}
   Let \(W\) be a spin (4k)-manifold with a metric \(h\) of positive scalar curvature which is a product metric on a collar neighborhood of \(\partial W\).  If \(\partial W=M\) has vanishing rational Pontryagin classes and \(g=h|_{M}\) has positive scalar curvature then

\begin{equation}\label{inv}
s(M,g)=-\left<j^{-1}(\hat{A}(TW)+a_kL(TW)),[W,\partial W]\right>+a_k\ \text{\normalfont sign}(W)   
\end{equation}
 where \([W,\partial W]\) is the fundamental class and \(\hat{A}\) and \(L\) are Hirzebruch's polynomials. Furthermore \(j^{-1}p_i(W)\) is any preimage of the \(i^{\text{th}}\) Pontryagin class of \(W\) in \(H^{4i}(W,\partial W;\f{Q})\) and \(\text{\normalfont sign}(W)\) is the signature of \(W\).
  
\end{thm}      

In \tref{thm} and \tref{cthm} part (b) the associated disc bundle to the sphere bundle is a spin manifold and hence we can apply \tref{ksinv}.  If however the disc bundle is not a spin manifold we use a different strategy.  In the special case of an \(S^1\) bundle with geodesic fibers, Kreck and Stolz use a cobordism argument to reduce to a case where another bounding manifold can be found and derive a correction term.  

\begin{thm}\label{ns}\cite{KS}
Let \( \pi:M\to B\) be a principal \(S^1\) bundle such that \(M\) is a  spin \((4k-1)\)-manifold with vanishing rational Pontryagin classes.  Suppose \(B\) is a spin manifold and \(M\) is given  the spin structure induced by the vector bundle isomorphism \(TM\cong \pi^*(TB)\oplus V\), where \(V\) is the trivial vector bundle generated by the action field of the \(S^1\) action.  Let \(g\) be a metric with scal\((g)>0\) on \(M\) such that \(S^1\) acts by isometries and the \(S^1\) orbits are geodesics.  Then

\begin{equation}\label{nsinv}
s(M,g)=-\left<j^{-1}({\hat{A}(TW)}\cosh(e/2)+a_k{L(TW)}),[W,\partial W]\right>+a_k\text{\normalfont sign}(W).
\end{equation}

Here \(W\) is the disc bundle associated to \(M\).  Furthermore \(j^{-1}\), \(\hat{A}\), \(L\), \(a_k\), \([W,\partial W]\) and \(\text{\normalfont sign}\)\((W)\) are as in \tref{ksinv} and \(e\in H^2(W,\f{Z})\)  is the image of the Euler class of the \(S^1\) bundle under the isomorphism \(H^2(W,\f{Z})\cong H^2(B,\f{Z}).\) \end{thm} 

Specializing to dimension 7, the formulas in \eqref{inv} and \eqref{nsinv} become 
\begin{equation}\label{inv2}
s(M,g)=-\frac{1}{2^7\cdot 7}p_1^2+\frac{1}{2^5\cdot 7}\ \text{sign}(W)
\end{equation}
and 
\begin{equation}\label{nsinv2}
s(M,g)=-\ov{2^7\cdot 7}p_1^2+\ov{2^7\cdot 3}\left(2p_1e^2-e^4\right)+\ov{2^5\cdot 7}\text{sign}(W)
\end{equation}
where \[p_1^2=\left<(j^{-1}(p_1(TW)^2),[W,\partial W]\right>, p_1e^2=\left<(j^{-1}(p_1(TW)e^2),[W,\partial W]\right> \text{ and } e^4=\left<(j^{-1}(e^4),[W,\partial W]\right>.\]

\smallskip

We will use the diffeomorphism classification of \cite{KS1} to prove \tref{cthm}.  It applies to  spin 7-manifolds with \(\pi_1(M)=0,\ H^2(M,\f{Z})=\f{Z},\ H^3(M,\f{Z})=0\) and \(H^4(M,\f{Z})\) finite cyclic and generated by the square of a generator of \(H^2(M,\f{Z})\). For such manifolds Kreck and Stolz defined three invariants \(s_1(M),s_2(M),s_3(M)\in\f{Q}/\f{Z}\) and proved that two such manifolds \(M\) and \(M'\) are diffeomorphic if and only if \(|H^4(M,\f{Z})|=|H^4(M',\f{Z})|\) and \(s_i(M)=s_i(M')\) for \(i=1,2,3\) (\cite{KS1} Theorem 3.1).   

\smallskip

 \section{Metrics on sphere and disc bundles }
 
 \bigskip
 In \cite{GZ} and \cite{GZ2} one finds many examples of metrics with nonnegative sectional curvature on principal \(SO(n)\) bundles such that \(SO(n)\) acts by isometries.  Hence the associated sphere bundles admit such metrics as well. We will apply \tref{ksinv} to appropriate metrics constructed on the associated sphere and disc bundles.    
 \begin{thm}\label{propB}
Let \(P\) be a principal \(SO(n+1)\) bundle admitting a metric \(g_P\), invariant under the \(SO(n+1)\) action, with sec\((g_P)\geq0\).  In the case \(n=1\) assume in addition that at each point \(x\in P\) there exists a 2-plane \(\sigma_x\subset T_xP\) with \(\text{ sec}_{g_P}(\sigma_x)>0\) which is orthogonal to the orbit of \(SO(2)\).  Then there exists a metric \(g_M\) on the associated sphere bundle \(M=P\times_{SO(n+1)}S^{n}\) with sec\((g_M)\geq0\) and scal\((g_M)>0\) that extends to a metric \(g_W\) on the associated disc bundle \(W=P\times_{SO(n+1)}D^{n+1}\) with scal\((g_W)>0\).  Furthermore \(g_W\) is a product near the boundary of \(W\).    

\end{thm}
 
 \begin{proof}
 
Let \(g_{S^n}\) be the standard metric on the sphere of radius 1/2.  We define the metric \(g_M\) such that the product metric \(g_P+g_{S^{n}}\) and \(g_M\) make the projection 
\[\rho:P\times S^n\to P\times_{SO(n+1)}S^n= M\]  
into a Riemannian submersion.  By the O'Neill formula \(g_M\) has nonnegative sectional curvature.

To show \(g_M\) has positive scalar curvature we must check that each point of \(M\) has a 2-plane of positive sectional curvature.  First assume \(n>1\).  Consider \((p,x)\in P\times S^n\).  Let \(X,Y\in\mathfrak{so}(n+1)\) such that the action fields \(X^*,Y^*\in T_xS^n\) are linearly independant.  The vertical space of the \(SO(n+1)\) action on \(P\times S^n\) is the set of vectors \((Z^*,-Z^*)\) for all \(Z\in\mathfrak{so}(n+1)\), where we repeat notation for the action fields on \(P\) and \(S^n\).  It follows that  projections of \((0,X^*),(0,Y^*)\in T_{(p,x)}P\times S^n\) onto the horizontal space  are \(A=(aX^*,bX^*)\) and \(B=(cY^*,dY^*)\) for some \(a,b,c,d\neq 0\).  In the product metric we have 
\[|A\wedge B|^2_{g_P+g_{S^{n}}}\text{sec}_{g_P+g_{S^{n}}}(A\wedge B )=|aX^*\wedge cY^*|^2_{g_P}\text{sec}_{g_P}(X^*\wedge Y^*)+|bX^*\wedge dY^*|^2_{g_{S^n}}\text{sec}_{g_{S^{n}}}(X^*\wedge Y^*)>0.\]  Since the plane is horizontal the O'Neill formula implies that \(\text{sec}_{g_M}(\rho_*A\wedge \rho_*B)>0.\)  Thus \(g_M\) has positive scalar curvature.  

In the case of \(n=1\) we have by assumption a 2-plane \(\sigma_x\subset T_xP\) in the horizontal space of the \(SO(2)\) action on \(P\).  It follows that \((\sigma_x,0)\) lies in the horizontal space of the \(SO(2)\) action on \(P\times S^1\), and by the O'Neill formula \(\text{sec}_{g_M}(\rho_*(\sigma_x,0))\geq\text{sec}_{g_P}(\sigma_x)\).  So \(g_M\) has a 2-plane of positive sectional curvature at each point and hence scal\((g_M)>0\).

We next show that \(g_M\) extends to a metric \(g_W\) on \(W\) with positive scalar curvature.  Let \(f:[0,1]\to\ar{}\) be a concave function with \(f(0)=0\), \(f'(0)=1,\) \(f'(r)<1\) for \(r\in (0,1]\) and   \(f(r)=1/2\) for \(r\in[R,1]\) for some \(R<1\) .  Then 
\[g_{D^{n+1}}=dr^2+f(r)^2g_{S^n}\]
is a smooth metric on \(D^{n+1}\) with \(\text{sec}({g_{D^{n+1}}})\geq0\).  Define the metric \(g_W\)  on \(W\) such that \(g_P+g_{D^{n+1}}\) and \(g_W\) make the projection 
\[\pi:P\times D^{n+1}\to P\times_{SO(n+1)}D^{n+1}=W \]
into a Riemannian submersion.  

The assumption that \(f\) is concave and \(f'(r)<1\) when \(r>0\) ensure sec\((g_{D^{n+1}})\geq0\).  Furthermore, when \(n>1\), planes tangent to the spheres of constant radius \(r\) will have positive sectional curvature, and we repeat the argument above to conclude \(\text{scal}(g_W)>0.\)  For \(n=1\), the same argument as for \(g_M\) implies scal\((g_W)>0\).     

 For \(r\in[R,1]\) the projection \(\pi\) can be regarded as
\[\pi:(P\times S^n)\times [R,1]\to (P\times_{SO(n+1)} S^n)\times [R,1]\cong M\times[R,1].\]  
The image is a collar neighborhood of the boundary of \(W.\)  Since \(f=1/{2}\) in this region, the metric on the left is \(g_P+g_{S^n}+dr^2\) and the metric induced on the quotient is \(g_M+dr^2\).  So \(g_W\) is a product metric near the boundary with \(g_W|_{\partial W}=g_M\).

\end{proof}

We note that by replacing \(g_{S^n}\) by \(\ov{\lambda}g_{S^n}\) in the proof and considering \(\lambda\in[0,1]\), one sees that \(g_M\) lies in the same path component of \(\mathfrak{M}_{\text{\normalfont sec}\geq0}(M)\) as the metric induced by \(g_P\) under the submersion \(P\to P/SO(n)\cong M\).

In the case of an \(S^1\) bundle with totally geodesic fibers, \tref{ns} applies without requiring the disc bundle to be spin.  The following theorem shows that some \(S^1\) invariant metrics with nonnegative sectional curvature can be deformed to metrics with geodesic fibers while maintaining positive scalar curvature.

\begin{thm}\label{path}
Let \(M\) be a manifold admitting a free \(S^1\) action and a metric \(g\) of nonnegative sectional curvature, invariant under that action. Suppose that for  each \(x\in M\) there is a 2-plane \(\sigma_x\subset T_xM\) orthogonal to the \(S^1\) orbit with \(\text{sec}(\sigma_x)>0\).  Then \(M\) admits a metric \(h\) of positive scalar curvature such that \(S^1\) acts is by isometries, the \(S^1\) orbits are geodesics, and \(h\) and \(g\) are in the same path component of \(\mathfrak{M}_{\text{\normalfont scal}>0}(M)\).  
\end{thm}  

\begin{proof}
Since the set of 2-planes orthogonal to the \(S^1\) orbits is compact, the maximum sectional curvature of such a plane at each point is a positive continuous function, and hence there exists \(C>0\) such that we can choose \(\sigma_x\) with sec\((\sigma_x)>C\). Let \(X\) be the action field of the \(S^1\) action on \(M\) and \(u=|X|_g\).  We fix \(0<\ep<\inf_M(u)\) such that
\[\sup_{x\in M,\ |Y|_g=1}\left(\ep^2\left|\frac{3Y(u)^2}{(u^2-\ep^2)^2}-\frac{\text{Hess}_u(Y,Y)}{u(u^2-\ep^2)}\right|\right)<\frac{C}{n-1}.\]
 For each \(\lambda\in(0,1]\) we define \(v:M\to\ar{}_{>0}\) by \[v_\lambda={\frac{\ep u}{\lambda\sqrt{u^2-\ep^2}}}\]
and a warpped product metric \(g_\lambda\) on  \(M\times S^1\): 
\[g_\lambda=g+v_\lambda^2d\theta^2.\] 
Next define the metric \(h_\lambda\) on \(M\) such that \(g_\lambda\) and \(h_\lambda\) make the projection 
\[\pi:M\times S^1\to M\times_{S^1}S^1\cong M\]
into a Riemannian submersion.  The action
\[z\cdot(x,y)\to(x,yz)\]
of \(S^1\) on \(M\times S^1\) is by isometries of \(g_\lambda\), commutes with the quotient action, and induces an action on \(M\times_{S^1}S^1\) which makes the diffeomorphism \(M\times_{S^1}S^2\cong M\) equivariant.  Thus \(S^1\) acts on \(M\) by isometries of \(h_\lambda\).    

We now show that scal\((h_\lambda)>0\) for all \(\lambda\in(0,1]\).  For a point \(x\in M \) let \(\{X_1,...,X_{n-1},X\}\) be an orthogonal basis of \(T_xM\) with \(|X_i|_g=1\) and \(\sigma_x=\text{span}(X_1,X_2)\).  Then we can find \(a,b\in\ar{}\) and \(Z=(aX,b\partial_\theta)\) such that \(\{(X_1,0),...,(X_{n-1},0),Z\}\) is an orthonormal basis of the horizontal space of \(\pi\) at \((x,y)\in M\times S^1\).  By the O'Neill formula  
\[\text{scal}(h_\lambda)\geq\text{sec}_{g_\lambda}((X_1,0)\wedge(X_2,0))+\sum_{(i,j)\neq(1,2)}\text{sec}_{g_\lambda}((X_i,0)\wedge(X_j,0))+\sum_{k=1}^{n-1}\text{sec}_{g_\lambda}((X_k,0)\wedge Z).\]       

Since \((M\times \{y\},g)\) is totally geodesic,    
\[\text{sec}_{g_\lambda}((X_1,0)\wedge(X_2,0))=\text{sec}_{g}(\sigma_x)>C\]
and
\[\text{sec}_{g_\lambda}((X_i,0)\wedge(X_j,0))=\text{sec}_{g}(X_i\wedge X_j)\geq0.\]

Furthermore, using the basis \(\{(X_k,0),\ov{b}Z\}\) for the plane \((X_k,0)\wedge Z\) \[\text{sec}_{g_\lambda}((X_k,0)\wedge Z)=\frac{\left<R_{g}(X_k,\frac{a}{b}X)\frac{a}{b}X,X_k\right>_{g}+\left<R_{g_\lambda}((X_k,0),(0,\pth))(0,\pth),(X_k,0)\right>_{g_\lambda}}{\frac{a^2}{b^2}u^2+v_\lambda^2} \]
\[\geq-\left|\text{sec}_{g_\lambda}((X_k,0)\wedge(0,\pth))\right|=-\left|\ov{v_\lambda}\text{Hess}_{v_\lambda}(X_k,X_k)\right|.\]

For details on the sectional curvatures of a warped product, used in the last equality, see \cite{B} Section 9J.  Applying the definition of \(v_\lambda\) we have
\[\text{scal}(h_\lambda)\geq C-\sum_{k=1}^{n-1}\ep^2\left|\frac{3X_k(u)^2}{(u^2-\ep^2)^2}-\frac{\text{Hess}_u(X_k,X_k)}{u(u^2-\ep^2)}\right|>0.\]

So \(h_\lambda\), \(\lambda\in(0,1]\), is a continuous path of metrics with positive scalar curvature.  Each \(h_\lambda\) is identical to \(g\) on the orthogonal complement of \(X\), while 
\[|X|^2_{h_\lambda}=\frac{u^2v_\lambda^2}{u^2+v_\lambda^2}=\frac{\ep^2u^2}{\lambda^2(u^2-\ep^2)+\ep^2}.\]

Since \(X\) is a Killing vector field and \(|X|_{h_1}=\ep\) is constant, the integral curves of \(X\), which are the orbits of the \(S^1\) action, are geodesics in \(h_1\).  Furthermore \(|X|_{h_0}=u\), so \(h_0=g\).  Thus \(h=h_1\) and \(g\) are in the same path component of \(\mathfrak{M}_{\text{scal}>0}(M)\).

\end{proof}

In Sections 3 and 4 we find, for each manifold \(M\) in Theorems A and B, a sequence of metrics on manifolds diffeomorphic to \(M\) such that no two metrics yield the same value of \(s\).  The following lemma shows that these sequences complete the proof of Theorems A and B.  

  \begin{lem}\label{comps}
  Let \(M\) be a simply connected spin (4k-1)-manifold with vanishing rational Pontryagin classes.  Let \(\{M_i,g_i\}_i\) be a sequence of Riemannian manifolds diffeomorphic to \(M\) such that sec\((g_i)\geq0\), scal\((g_i)>0\), and the values \(s(M_i,g_i)\) are all distinct.  Then \(\mathfrak{M}_{\text{\normalfont sec}\geq0}(M)\) and \(\mathfrak{M}_{\text{\normalfont Ric}>0}(M)\) have infinitely many path components.   
  
  \end{lem}
  
  \begin{proof}
We  pull back the metrics \(g_i\) to a sequence of metrics on \(M\).  By \cite{KS} Proposition 2.13,  \(s\) is preserved under pullbacks.  Since the values of \(s\) are distinct, these metrics lie in an infinite set of distinct path components of \(\mathfrak{M}_{\text{\normalfont scal}>0}(M)\).

 An argument as in \cite{DKT, BKS} shows the metrics lie in different path components of \(\mathfrak{M}_{\text{sec}\geq0}\) as well.  Suppose two of the metrics \(g_0,g_1\) (up to diffeomorphism) can be connected with a path \(g_t\) maintaining nonnegative sectional curvature.  Bohm and Wilking \cite{BW} showed that such metrics on a simply connected manifold immediately evolve to have positive Ricci curvature under the Ricci flow.  Thus the path \(g_t\) evolves to a path maintaining positive Ricci curvature, and thus positive scalar curvature.  \(g_0\) and \(g_1\) are connected to the new endpoints by their evolution under the Ricci flow which similarly maintains positive scalar curvature.  So \(g_0\) and \(g_1\) can be connected with a path maintaining positive scalar curvature, which leads to a contradiction.  Furthermore, \(g_0\) and \(g_1\) evolve under the Ricci flow to metrics of positive Ricci curvature in distinct components of \(\mathfrak{M}_{\text{\normalfont scal}>0}(M)\) and therefore of \(\mathfrak{M}_{\text{\normalfont Ric}>0}(M).\)   
\end{proof}
\section{\(S^3\) bundles over \(S^4\) and \(\f{C}P^2\)}
\bigskip
In this section we prove \tref{thm}, starting with the simplest case.  
\subsection{\(S^3\) bundles over \(S^4\)}\(S^3\) bundles over \(S^4\) are classified by elements of \(\pi_3(SO(4))=\f{Z}\oplus\f{Z}\).  We use the basis for \(\pi_3(SO(4))\) given by the maps \(\mu(q)(v)=qvq^{-1}\) and \(\nu(q)(v)=qv\).  Here \(v\in\ar{4}\) viewed as the  quaternions and \(q\in S^3\) viewed as the unit quaternions.  Let \(M_{m,n}\) be the bundle classified by \(m\mu+n\nu\in\pi_3(SO(4))\).  In \cite{GZ} it is shown that the \(SO(4)\) principal bundle of every \(S^3\) bundle over \(S^4\) admits an \(SO(4)\) invariant metric of nonnegative sectional curvature, and hence the sphere bundles do as well.  

Assume \(n\neq0\).  From the homotopy long exact sequence one sees that \(H^4(M_{m,n},\f{Z}_n)=\f{Z}_n\), so the rational Pontryagin classes of \(M_{m,n}\) vanish.  Let \(W_{m,n}\) be the associated disc bundle.  Then \(H^2(W_{m,n},\f{Z}_2)=H^2(S^4,\f{Z}_2)=0\) and hence \(W_{m,n}\) is a spin manifold.  \tref{propB} and \tref{ksinv} imply that  \(M_{m,n}\) has a metric \(g_{M_{m,n}}\) of nonnegative sectional and positive scalar curvature with \(s\) invariant given by \eqref{inv2}.  Crowley and Escher \cite{CE} computed the invariants 
\[p_1^2(W_{m,n})=\frac{4(n+2m)^2}{n}\] 
and 
\(\text{sign}(W_{m,n})=1\).
So 
\[s(M_{m,n},g_{M_{m,n}})=\frac{-(n+2m)^2+n}{2^5\cdot 7\cdot n}.\]

Corollary 1.6 of \cite{CE} shows that \(M_{m',n}\) and \(M_{m,n}\) are diffeomorphic if \(m'=m\) mod \(56n\).  So the manifolds in the sequence \(\{M_{m+56ni,n}\}_i\) are all diffeomorphic to \(M_{m,n}\).  Since \(n\) is constant in the sequence, the \(s\) invariant is a polynomial in \(i\).  It follows that there is an infinite subsequence of metrics with distinct \(s\) invariants.  \lref{comps} completes the proof of the first part of \tref{thm}.      

\begin{rem}

Comparison to the 7-dimensional homogeneous spaces in \cite{N} shows that \(M_{m,n}\) has the cohomology ring of such a space only when \(|n|=1\), \(2\) or \(10\).  The homogeneous candidates are \(S^7\), \(T_1S^4\) and the Berger space \(SO(5)/SO(3)\) with \(H^4=0, \f{Z}_2\) and \(\f{Z}_{10}\).      

\end{rem}

\subsection{\(S^3\) bundles over \(\f{C}P^2\)}
In \cite{GZ2} it is shown that every principal \(SO(4)\) bundle over \(\f{C}P^2\) with \(w_2\neq0\) admits an \(SO(4)\) invariant metric of nonnegative sectional curvature.  Such bundles are classified by two integers \(a,b\) describing the first Pontryagin and Euler classes \(p_1=(2a+2b+1)x^2\) and \(e=(a-b)x^2\), where \(x\) is the generator of \(H^*(\f{C}P^2, \f{Z})\).  Let \(\pi:S_{a,b}\to\f{C}P^2\) be the \(S^3\) bundle over \(\f{C}P^2\) with these characteristic classes.  If \(a\neq b\) then the Gysin sequence implies that \(H^4(S_{a,b},\f{Z})=\f{Z}_{|a-b|},\) so the rational Pontryagin classes vanish.

Let \(E^4\xrightarrow{} \f{C}P^2\) be the 4-plane bundle associated to \(S_{a,b}\)  and \(W_{a,b}\subset E^4\) the associated disc bundle with projection \(\rho:W_{a,b}\to\f{C}P^2\).  Then \(TW_{a,b}\cong\rho^*(E^4\oplus T\f{C}P^2)\) and \(w_2(TW_{a,b})=\rho^*(w_2(E^4)+w_2(T\f{C}P^2))=0.\)  So \(W_{a,b}\) is a spin manifold.  \tref{propB} and \tref{ksinv} imply that \(S_{a,b}\) has a metric \(g_{S_{a,b}}\) of nonnegative sectional and positive scalar curvature with \(s\) invariant given by \eqref{inv2}.  

It is shown in \cite{EZ} Proposition 4.3 that 

\[p_1^2(W_{a,b})=\ov{a-b}(2a+2b+4)^2\]
and
\[\text{sign}(W_{a,b})=\text{sgn}(a-b).\]
So

\[s(S_{a,b},g_{S_{a,b}})=-\frac{(a+b+2)^2}{2^5\cdot 7\cdot (a-b)}+\frac{\text{sgn}(a-b)}{2^5\cdot 7}.\]

Corollary 4.5 of \cite{EZ} implies that \(S_{a,b}\) and \(S_{a',b'}\) are diffeomorphic if \(a-b=a'-b'\) and \(a=a'\) mod \(\lambda=2^3\cdot 3\cdot 7\cdot |a-b|.\)  Thus the manifolds in the sequence \(\{S_{a+i\lambda,b+i\lambda}\}_{i}\) are all diffeomorphic to \(S_{a,b}\).  Since \(a-b\) is constant for the sequence, the \(s\) invariant is a polynomial in \(i\).  So there is an infinite subsequence of metrics in this sequence with distinct \(s\) invariants.  \lref{comps} completes the proof of the second part of \tref{thm}.  

\begin{rem}\label{hom}

(a) The only 7-dimensional homogeneous spaces with the same cohomology ring as any \(S_{a,b}\) are the families \(N^7_{k,l}\) and \(W^7_{k,l}\) described in the introduction, see \cite{N}.  The quantities \(|H^4(N^7_{k,l},\f{Z})|=l^2\) and \(|H^4(W^7_{k,l},\f{Z})|=k^2+l^2+kl\) are always equal to \(0\) or \(1\) mod \(3\), so if \(|a-b|=2\) mod \(3\), \(S_{a,b}\) does not have the homotopy type of a 7-dimensional homogeneous space.

(b)  By \cite{EZ} Proposition 6.7, \(S_{-1,a(a-1)}\) is diffeomorphic to the homogeneous Aloff-Wallach space \(W^7_{a,1-a}\).  There also exist infinitely many positively curved Eschenburg spaces and many other Aloff-Wallach spaces which are diffeomorphic to \(S^3\) bundles over \(\f{C}P^2\), see \cite{EZ}  Theorem 8.1.
 
\end{rem}

\section{\(S^1\) bundles over \(S^2\) bundles over \(\f{C}P^2\)}

\bigskip

Escher and Ziller \cite{EZ} defined two families of 7-manifolds as follows.  Let \(x\) be the generator of \(H^*(\f{C}P^2,\f{Z})\).  
Define \(p:{N}_t\to\f{C}P^2\) as the \(S^2\) bundle with Pontryagin and Stiefel-Whitney classes \(p_1(N_t)=(1-4t)x^2\) and \(w_2(N_t)\neq0\).  They showed that  \(N_t\) is diffeomorphic to the projectivization \(P(E_t)\) of the complex line bundle \(E_t\) over \(\f{C}P^2\) with Chern classes \(c_1(E_t)=x\) and \(c_2(E_t)=tx^2\).  Furthermore if \(P_t\) is the principal \(U(2)\) bundle corresponding to \(E_t\), \(N_t\) is diffeomorphic to \(P_t/T^2\), where \(T^2\subset U(2)\).  

Let \(y\) be the first Chern class of the dual of the tautological line bundle over \(P(E)\).  By the Leray-Hirsch theorem

\[H^*(N_t)=\f{Z}[x,y]/(x^3, y^2+xy+tx^2).\]
For simplicity, we denote \(p^*(x)\) again by \(x\).  Finally define the principal \(S^1\) bundle

 \[S^1\rightarrow M^t_{a,b}\xrightarrow{}N_t \indent\indent \text{with Euler class} \indent e=ax+(a+b)y\indent \text{and gcd}(a,b)=1.\]  Proposition 6.1 in \cite{EZ} shows that the bundle \(S^1\to M^t_{a,b}\to N_t\) is equivalent to \(T^2/S^1_{a,b}\to P_t/S^1_{a,b}\to P_t/T^2\cong N_t\) where \(S^1_{a,b}=\{\text{diag}(e^{ia\theta},e^{ib\theta})\}\subset U(2)\).  Since gcd\((a,b)=1\), the total space is simply connected, and from the Gysin sequence it follows that the cohomology ring of \(M^t_{a,b}\) is of the form required by the diffeomorphism classification of \cite{KS1} as long as \(0\neq |t(a+b)^2-ab|=|H^4(M_{a,b}^t,\f{Z})|\).

Next define \(\bar{p}:\bar{N}_t\to\f{C}P^2\) as the \(S^2\) bundle with Pontryagin and Stiefel-Whitney classes \(p_1(\bar{N}_t)=4tx^2\) and \(w_2(\bar{N}_t)=0\).   
In this case \(\bar{N}_t\) is diffeomorphic to the projectivization \(P(\bar{E}_t)\) of the complex line bundle \(\bar{E}_t\) over \(\f{C}P^2\) with Chern classes \(c_1(\bar{E}_t)=2x\) and \(c_2(\bar{E}_t)=(1-t)x^2\).  If \(\bar{P}_t\) is the principal \(U(2)\) bundle associated to \(\bar{E}_t\),  \(\bar{N}_t\) is diffeomorphic to \(\bar{P}_t/T^2\).  Let \({y}\) be the first Chern class of the dual of the tautological line bundle over \(P(\bar{E}_t)\).  Then

\[H^*(\bar{N}_t)=\f{Z}[x,y]/(x^3, {y}^2+2xy+(1-t)x^2).\]
Again we denote \(\bar{p}^*(x)\) by \(x\).  Finally define the principal \(S^1\) bundle

 \[S^1\rightarrow \bar{M}^t_{a,b}\xrightarrow{}\bar{N}_t \indent\indent \text{with Euler class} \indent e=(a+b)x+b{y}\indent \text{and gcd}(a,b)=1.\]  
 
 In this case, one sees that \(\pi_1(\bar{P}_t)=\f{Z}_2\) and \(\bar{P}_t\) has a two-fold cover \(\bar{P}_t'\) which is a principal \(S^1\times S^3\) bundle over \(\f{C}P^2\).  Furthermore \(\bar{N}_t\cong \bar{P}_t'/T^2\), with \(T^2=\{(e^{i\theta},e^{i\phi})\}\subset S^1\times S^3\).  Proposition 7.5 in \cite{EZ} shows that the bundle defining \(\bar{M}^t_{a,b}\) is equivalent to \(T^2/S^1_{-b,a}\to \bar{P}_t'/S^1_{-b,a}\to\bar{P}_t'/T^2\) where \(S^1_{-b,a}=\{(e^{-ib\theta},e^{ia\theta})\}\).  As before, \(\bar{M}^t_{a,2b}\) is simply connected and has the cohomology necessary for the diffeomorphism classification of \cite{KS1}, since \(a\) is odd and so \(|H^4(M^t_{a,2b})|=|a^2-4tb^2|\neq0\).

 Escher and Ziller showed that \(M^t_{a,b}\) and \(\bar{M}^{2t}_{a,b}\) admit \(S^1\) invariant metrics \(g^t_{a,b}\) and \(\bar{g}^{2t}_{a,b}\) respectively with nonnegative sectional curvature.  In order to apply \tref{propB} and \tref{path} we prove the following lemma.      
 
 \begin{lem}\label{scal}
At each point \(x\) of \((M^t_{a,b},g^{t}_{a,b})\) and \((\bar{M}^{2t}_{a,b},\bar{g}^{2t}_{a,b})\) there exists a 2-plane \(\sigma_x\) orthogonal to the \(S^1\) orbit with sec\((\sigma_x)>0\).  
\end{lem} 
\begin{proof}
 
 The metrics are constructed using cohomogeneity one actions, and we first recall the general description of such manifolds.  We consider actions of a compact Lie group \(G\) on a manifold \(M\) such that the orbit space is the interval \([-1,1]\).  Let \(\pi:M\to[-1,1]\) be the projection onto the orbit space.   Let \(H\subset G\) be the isotropy subgroup of a point in the principal orbit \(\pi^{-1}(0)\) and \(K^\pm\)  the isotropy groups of points in the singular orbits \(\pi^{-1}(\pm 1)\).  The slice theorem implies that \(\pi^{-1}([-1,0])\) is equivariantly diffeomorphic to the disc bundle \(D_-=G\times_{K_-}D^{d_-}\) where \(K_-\) acts linearly on \(D^{d_-}\)  and \(K_-/H\) is diffeomorphic to the sphere \(S^{d_--1}\). Here \(d_-\) is the codimension of the singular orbit.  Furthermore, the boundary of \(D_-\) is \(G/H\), diffeomorphic to the principal orbit \(\pi^{-1}(0)\).  \(D_+\) is described equivalently with the same boundary.  Then \(M\) is diffeomorphic to the union \(D_-\cup_{G/H}D_+\).  Conversely, given Lie groups \(H\subset K_\pm\subset G\) with \(K_\pm/H\cong S^{d_\pm-1},\) the action of \(K_\pm\) on \(S^{d_\pm-1}\) extends to a linear action on \(D^{d_\pm}.\)  We can then define \(M=D_-\cup_{G/H}D_+\) as above, and \(M\) will admit a cohomogeneity one action by \(G\) with isotropy groups \(H\subset K_\pm\).     

If \(d_\pm=2\), it is shown in \cite{GZ} that one can define a metric with sec\(\ \geq0\) on \(M\) as follows.    Let \(\mathfrak{g,k,h}\) be the Lie algebras of \(G,K_-,H\) respectively and \(Q\) a biinvariant metric on \(G\).  Choose \(\mathfrak{g}=\mathfrak{m}\oplus\mathfrak{k}\) and \(\mathfrak{k}=\mathfrak{h}\oplus\mathfrak{p}\) to be \(Q\)-orthogonal decompositions and \(\overline{Q}_a\) the left invariant metric on \(G\) defined by 
\[\overline{Q}_a=Q|_{\mathfrak{m}\oplus\mathfrak{h}}+aQ|_\mathfrak{p}.\]  
Let \(f(r)\) be a concave function with \(f(0)=0\), \(f'(0)=1\) and \( f(r)=\sqrt{\frac{al^2}{a-1}}\) for  \(r\) near the boundary of \(D^2\), where \(2\pi l\) is the length of \(K_-/H\) with respect to \(Q\).  In \cite{GZ} it is shown that the metric
\[\overline{g}=\overline{Q}_a+dr^2+f(r)d\theta^2\]    
 on \(G\times D^2\) has nonnegative curvature as long as \(1< a\leq4/3\) and hence induces a \(G\) invariant metric \(g_-\) of nonnegative curvature on the quotient \(D_-\).  Furthermore  \(g_-\) is a product near the boundary \(G/H\), with the induced metric on \(G/H\) the same as that induced by \(Q\).  A similar metric can be put on \(D_+\), and because of the boundary condition the two can be glued to form a smooth \(G\) invariant metric \(g\) of nonnegative sectional curvature on \(D_-\cup_{G/H}D_+\cong M\).    

In order to prove the claim, we need to describe the manifolds and metrics in a slightly different way than in \cite{GZ2}.  For \(p_-,p_+,q\in\f{Z}\), \(p_+\) odd and \(p_-\neq q\) mod 2, let \(P_{p_-,p_+,q}\) be the cohomogeneity one manifold defined by the following Lie groups:
\[G=U(2)\times S^3\]
\[H=\f{Z}_4=\left<\left( \pm i^q\left[\begin{array}{cc}0&1\\-1&0\end{array}\right],j\right)\right>\]
 \[K_-=H\cdot\left\{\left(\text{diag}\left(e^{ip_-\theta},e^{-ip_-\theta}\right),e^{i\theta}\right)\right\}\]
 \[K_+=\left\{\left(e^{iq\theta}R(p_+\theta),e^{j\theta}\right)\right\}\]
where \(R(\phi)\) represents a \(2\times2\) rotation matrix and the sign in \(H\) is chosen to make \(H\) a subgroup of \(K_+\).  One easily sees that \(U(2)\) acts freely on \(P_{p_+,p_-,q}\).  Since \(U(2)\) commutes with \(S^3\), the quotient \(P_{p_+,p_-,q}/U(2)\) admits an action by \(S^3\) which is cohomogeneity one with the same isotropy groups as the action of \(S^3\) on \(\f{C}P^2\) (see \cite{GZ2} Figure 2.2.)  Thus \(P_{p_+,p_-,q}\) is the total space of a principal \(U(2)\) bundle over \(\f{C}P^2\).  

Suppose \(P\) is a principal \(U(2)\) bundle over \(\f{C}P^2\).  From the spectral sequence of the fibration \(U(2)\to P\to \f{C}P^2\), one sees that \(H^2(P,\f{Z})\cong\f{Z}_{|c_1|}\), where \(c_1\) denotes the coefficient of \(x\) in the first Chern class \(c_1(P)\in H^2(\f{C}P^2,\f{Z})\).  Applying the Seifert-Van Kampen theorem to \(P_{p_-,p_+,q}=D_-\cup D_+\), one shows that \(\pi_1(P_{p_-,p_+,q})=\f{Z}_q\).  By the universal coefficient theorem we conclude that \(H^2(P_{p_+,p_-,q},\f{Z})=\f{Z}_q\) and hence \(c_1(P_{p_+,p_-,q})=qx\).      

Let \(Z\) be the center of \(U(2)\).  Since \(U(2)/Z\cong SO(3)\), \(P/Z\) is a principal \(SO(3)\) bundle over \(\f{C}P^2\) with first Pontryagin class \(p_1(P/Z)=c_1(P)^2-4c_2(P)\), see \cite{EZ}, 2.5, 2.6.  In particular, \(P_{p_-,p_+,q}/Z\) admits a cohomogeneity one action by \(SO(3)\times S^3\) and one easily shows that the isotropy groups are 
\[H=\f{Z}_4=\left<\left(R_{1,3}(\pi),j\right)\right>\] 
\[K_-=H\cdot\left\{\left(R_{2,3}(2p_-\theta),e^{i\theta}\right)\right\}\]
\[K_+=\left\{\left(R_{1,3}(2p_+\theta),e^{j\theta}\right)\right\}.\]
Here \(R_{n,m}\in SO(3)\) is a rotation in the \(n,m\) plane of \(\ar{3}\).  By \cite{GZ2} Theorem 4.7, this bundle has first Pontryagin class \(p_1(P_{p_-,p_+,q}/Z)=(p_+^2-p_-^2)x^2\).  It follows that \(c_2(P_{2t})=\ov{4}(q^2-p_+^2+p_-^2)x^2\).

The description of the action on \(P_{p_-,p_+,q}\) has \(d_\pm=2\), so we can construct a \(U(2)\)-invariant metric \(g\) with sec\(\ \geq0\) as above.  We check that \(g\) has a 2-plane with sec\(\ >0\) orthogonal to the orbit of \(T^2\subset U(2)\) at each point.  We do this on each half \(D_\pm=G\times_{K_\pm}D^2\) separately.  By the O'Neill formula it is necessary to find such a 2-plane orthogonal to the orbit of \(T^2\times K_\pm\) at each point of \(G\times D^2\).  For \(D_-\) we have \(\mathfrak{g}=\mathfrak{u}(2)\oplus\mathfrak{su}(2)\), \(\mathfrak{k}=\mathfrak{p}=\text{span}\{(p_-i,i)\}\) and \(\mathfrak{h}=\{0\}\).  Here \(\{i,j,k\}\) is the standard basis of \(\mathfrak{su}(2)\) and \(\{l,i,j,k\}\) is the standard basis of \(\mathfrak{u}(2)\) with  \(l\) the generator of the center. 

 Since \(T^2\) and \(K_-\) act on \(G\) on the left and right respectively, the tangent space to the orbit at each point \((y,z)\in G\times D^2\) is contained in \[dL_y(\mathfrak{k})+dR_y(\mathfrak{u}(2)\oplus\{0\})+T_zD^2.\]   Here \(L_y\) and \(R_y\) designate left and right translation on \(G\). Since \((0,j)\) and \((0,k)\) are orthogonal to \(\mathfrak{k}\) and \(\mathfrak{u}(2)\oplus\{0\}\) with respect to the left invariant metric \(\overline{Q}_a\) and \(\mathfrak{u}(2)\oplus\{0\}\) is Ad-invariant, \(dL_y(0,j)\) and \(dL_y(0,k)\) are orthogonal to the orbit of \(T^2\times K_\pm\).  Choose \(\tau_{[y,z]}\) to be the image of \(dL_y(0,j)\wedge dL_y(0,k)\).
 By the O'Neill formula
\[\text{sec}_g(\tau_{[y,z]})\geq\frac{3}{4}|[dL_y(0,j),dL_y(0,k)]^V|_{\bar{g}}^2\geq{3}|dL_y(0,i)^V|_{\bar{g}}^2>0\]
where \(dL_y(0,i)^V\) is the projection of \(dL_y(0,i)\) onto \(dL_y(\mathfrak{k})\).  The same argument can be made on \(D_+\) using \(dL_y(0,i)\wedge dL_y(0,k)\).  

To summarize,  \(P_{p_-,p_+,q}\) is the \(U(2)\) principal bundle over \(\f{C}P^2\) with Chern classes \(c_1=qx\) and \(c_2=\ov{4}(q^2-p_+^2+p_-^2)x^2\) and it admits a \(U(2)\) invariant metric \(g\) and a 2-plane \(\tau_{\bar{x}}\) at each point \(\bar{x}\in P_{p_-,p_+,q}\) with \(\tau_{\bar{x}}\perp T^2\cdot \bar{x}\) and sec\(g(\tau_{\bar{x}})>0\).  In particular, \(P_{2t,1-2t,1}=P_t\) and \(P_{2t-1,2t+1,2}=\bar{P}_{2t}.\)  The metric \(g_{a,b}^t\) is defined such that \(g\) and \(g_{a,b}^t\) make \(P_t\to P_t/S^1_{a,b}\) into a Riemannian submersion.  Let \(g'\) be the locally isometric lift of \(g\) to the universal cover \(\bar{P}'_{2t}\) of \(\bar{P}_{2t}\).  Note that the \(T^2\subset U(2)\) action on \(\bar{P}_{2t}\) lifts to the \(T^2\subset S^1\times S^3\) action on \(\bar{P}'_{2t}\). \(\bar{g}_{a,b}^{2t}\) is defined such that \(g'\) and \(\bar{g}_{a,b}^{2t}\) make \(\bar{P}_{2t}'\to\bar{P}_{2t}'/S^1_{-b,a}\) into a Riemannian submersion.    

On each manifold, the image \(\sigma_x\) of \(\tau_{\bar{x}}\) under the \(S^1\) quotient will be orthogonal to the orbits of \(T^2/S^1\).  Using the O'Neill formula once more it follows that \(\text{sec}(\sigma_x)>0\) with respect to \(g^t_{a,b}\) and \(\bar{g}^{2t}_{a,b}\).  We  note that these metrics are invariant under the centralizer of \(S^1\), which is isomorphic to \(S^1\times S^3\) in each case.  The groups acting effectively by isometries are \(S^1\times SO(3)\) and \(U(2)\) respectively.

  \end{proof}
\lref{scal} yields the metrics required to calculate the \(s\) invariant for the two families of \(S^1\) bundles.

\subsection*{\(S^1\) bundles over spin \(S^2\) bundles}
 Let \(\pi:\bar{E}^3\to\f{C}P^2\) be the 3-plane bundle associated to \(S^2\to \bar{N}_t\xrightarrow{\bar{p}} \f{C}P^2\) and \(i\) the inclusion \(i:\bar{N}_t\to \bar{E}^3\).  Then \(T\bar{N}_t\) and the normal bundle \(V\) to \(\bar{N}_t\) span \(i^*(T\bar{E}^3)\).  Since \(T\bar{E}^3\cong \pi^*(\bar{E}^3\oplus T\f{C}P^2)\) and \(V\) is trivial we have
 \[w_2(T\bar{N}_t)=\bar{p}^*(w_2(\bar{E}^3)+w_2(\f{C}P^2))=x.\] 
 
   Next let \(E^2\) be the 2-plane bundle associated to \(\bar{M}^t_{a,b}\) and \(\bar{W}^t_{a,b}\subset E^2\) the disc bundle with projection \(\sigma:\bar{W}^t_{a,b}\to\bar{N}_t\).  We have the bundle isomorphism \(T\bar{W}^t_{a,b}\cong \sigma^*(E^2\oplus T\bar{N}_t)\)  and second Stiefel Whitney class \[w_2(T\bar{W}^t_{a,b})=w_2(E^{2})+w_2(T\bar{N}_t)=e(\bar{M}^t_{a,b})+w_2(T\bar{N}_t)=(a+b+1)x+by \ \ \text{ mod }\ \ 2.\] 
Here the notation \(\sigma^*\) is repressed since it is an isomorphism on cohomology.  Thus when \(b\) is even, \(a\) is odd, and \(\bar{W}^t_{a,b}\) is a spin manifold.  From the Gysin sequence one sees that \(H^4(\bar{M}^t_{a,b},\f{Z})\) is torsion so all the rational Pontryagin classes vanish.

   Therefore for \(t\) and \(b\) even, using \lref{scal} \((\bar{M}_{a,b}^{t},\bar{g}^{t}_{a,b})\) satisfies the hypotheses of \tref{propB} and \tref{ksinv}, and \(s(\bar{M}_{a,b}^t,\bar{g}^t_{a,b})\) is given by \eqref{inv2}.   

  In \cite{EZ} it was shown that for \(\bar{W}_{a,b}^t\) we have
\[p_1^2=b\left(-\frac{(3+4t)^2}{a^2-b^2t}+6+8t+3a^2+tb^2\right)\]
and \[\text{sign}(\bar{W}^t_{a,b})=\left\{\begin{array}{cc}0&\text{ if \(a^2-tb^2>0\)} \\ 2&\text{  if \(a^2-tb^2<0\) and \(b(1+t)>0\)} \\ -2&\text{ if \(a^2-tb^2<0\) and \(b(1+t)<0\)}\end{array}\right..\]
 Thus for \(t\) and \(b\) even
\begin{equation}\label{inveq}
 s(\bar{M}_{a,b}^t,\bar{g}_{a,b}^t)=\frac{b}{2^7\cdot7}\left(\frac{(3+4t)^2}{a^2-b^2t}-6-8t-3a^2-tb^2\right)+\frac{1}{2^5\cdot 7}\text{sign}(\bar{W}_{a,b}^t).
 \end{equation}
 
 When \(b\) is even, \(\bar{M}^{t}_{a,b}\) has the cohomology appropriate to calculate the Kreck-Stolz diffeomorphism invariants \(s_1,s_2,s_3\in\f{Q}/\f{Z}\)  \cite{KS1}.  These invariants are calculated in \cite{EZ}:   

\bigskip

\indent \(\displaystyle s_1(\bar{M}_{a,b}^t)=s(\bar{M}^t_{a,b},g^t_{a,b})\ \text{ mod }\f{Z}\)

\smallskip

\indent \( \displaystyle s_2(\bar{M}_{a,b}^t)=-\frac{1}{2^4\cdot 3} (b(n^2+tm^2)-2anm)\)
\[ 
-\frac{1}{2^4\cdot 3\cdot (a^2-tb^2)} (4nm(an^2+atm^2+2tbnm)-(3+4t-2n^2-2tm^2)(bn^2+btm^2+2anm))\]

\smallskip 

and

\smallskip

\indent \(\displaystyle s_3(\bar{M}_{a,b}^t)=-\frac{1}{2^2\cdot3}(b(n^2+tm^2)-2anm)\)\[-\frac{1}{2^3\cdot3\cdot (a^2-tb^2)}(16nm(an^2+atm^2+2tbnm)-(3+4t-8n^2-8tm^2)(bn^2+btm^2+2anm))\]
where \(m,n\) are such that \(ma+nb=1\).  

\bigskip

Now set \(r=a^2-8tb^2\), \(\lambda=2^7\cdot 3\cdot 7r\) and choose \(m,n\) such that \(ma+2nb=1\).  We then define
\[a_k=a+16b^2\lambda k,\ b_k=b,\ t_k=t+4a\lambda k+32b^2\lambda^2k^2,\ m_k=m,\ n_k=n-8b\lambda mk.\]

We see that \(a_k^2-8t_kb_k^2=r,\) \(m_ka_k+2n_kb_k=1\), and each of \(a_k,b_k,m_k,n_k,t_k\) is equal to the corresponding \(a,b,m,n,t\) mod \(\lambda\).  When \(r<0\) we have \(t,t_k>0\), so \(2b_k(1+2t_k)\) has the same sign as \(2b(1+2t)\).  It follows that sign\((\bar{W}^{2t_k}_{a_k,2b_k})=\)sign\((\bar{W}^{2t}_{a,2b})\).  This is enough to ensure the numerators of \(s_i(M^{2t}_{a,2b})\) and \(s_i(M^{2t_k}_{a_k,2b_k})\) are equal modulo the denominators so \(s_i(M^{2t}_{a,2b})-s_i(M^{2t_k}_{a_k,2b_k})\in\f{Z}\).  Thus the invariants \(s_i\in\f{Q}/\f{Z}\) and \(|H^4(M,\f{Z})|\) are equal and \(M^{2t_k}_{a_k,2b_k}\) is diffeomorphic to \(M^{2t}_{a,2b}\) by \cite{KS1} Theorem 3.1.  Since \(a_k^2-8t_kb_k^2\) and sign\((\bar{W}^{2t_k}_{a_k,2b_k})\) are constant for the sequence \(\{M^{2t_k}_{a_k,2b_k}\}_k\), the \(s\) invariant is a polynomial in \(k\), and there is an infinite subsequence of metrics with distinct \(s\) invariants.  \lref{comps} completes the proof of \tref{cthm} part (b).

\subsection*{\(S^1\) bundles over non-spin \(S^2\) bundles}
Let \(\pi:E^3\to\f{C}P^2\) be the 3-plane bundle associated to  \(p:N_t\to\f{C}P^2\).  By the bundle isomorphism of the previous section we have
 \[w_2(TN_t)=p^*(w_2(E^3)+w_2(\f{C}P^2))=p^*(2x)=0\text{ mod }2.\]
 Thus we can give \(M_{a,b}^t\) the spin structure induced from the bundle isomorphism \(TM_{a,b}^t\cong\rho^*TN_t\oplus V'\), where \(V'\) is the bundle generated by the \(S^1\) action field and \(\rho:M^t_{a,b}\to N_t\).  From the Gysin sequence one sees that \(H^4(M^t_{a,b},\f{Z})=\f{Z}_{|t(a+b)^2-ab|}\), so the rational Pontryagin classes vanish when \(t(a+b)^2-ab\neq0\).  
 
 By \lref{scal} \(g_{a,b}^t\) satisfies the conditions of \tref{path}.  It follows that \(s(M^t_{a,b},g^t_{a,b})=s(M^t_{a,b},h)\) for an \(S^1\) invariant metric \(h\) with geodesic fibers.  Then the circle bundle \(M^t_{a,b}\to N_t\) and \(h\) satisfy the hypotheses of \tref{ns} and \(s(M^t_{a,b},g^t_{a,b})\) is given by \eqref{nsinv2}.      
 
In \cite{EZ} the terms \(p_1^2,\) \(p_1e^2\) and \(e^4\) are calculated for \(W^t_{a,b}\) and we have
 \[s(M_{a,b}^t,g_{a,b}^t)=\frac{(a+b)(1-t)^2}{2^3\cdot 7\cdot(t(a+b)^2-ab)}+\ov{2^5\cdot 3\cdot 7}(-3ab+(1-t)(8+(a+b)^2))+\ov{2^5\cdot 7}\text{sign}(W_{a,b}^t)\]

where 
 \[\text{sign}(W_{a,b}^t)=\left\{\begin{array}{cc}0&\text{ if \(ab-t(a+b)^2<0\)} \\ 2&\text{  if \(ab-t(a+b)^2>0\) and \(a+b>0\)} \\ -2&\text{ if \(ab-t(a+b)^2>0\) and \(a+b<0\)}\end{array}\right..\]
  
 When \(t(a+b)^2\neq ab\), \(M^t_{a,b}\) also has the cohomology ring necessary to define the diffeomorphism invariants \(s_i\).  They are calculated in \cite{EZ} Proposition 5.2. Just as for \(\bar{M}^t_{a,b}\) they are given by rational functions with numerators depending on \(a,b,m,n,t\) and sign(\(W^t_{a,b}\)), where \(m,n\) are such that \(ma+nb=1\).  The denominators divide \(2^5\cdot 3\cdot 7\cdot |t(a+b)^2-ab|\).  As these are the only relevant details, we omit the equations for brevity.      
     
Let \(r=t(a+b)^2-ab\neq0\), \(\lambda=2^5\cdot3\cdot7r\) and choose \(m,n\) such that \(ma+nb=1\).  We set 
\[a_k=a+(a+b)^2\lambda k,\ b_k=b-(a+b)^2\lambda k,\ t_k=t-(a-b)\lambda k-(a+b)^2\lambda^2k^2,\]\[m_k=m+(n-m)(a+b)\lambda k,\ n_k=n+(n-m)(a+b)\lambda k.\]

One checks that \(t_k(a_k+b_k)^2-a_kb_k=r\), \(m_ka_k+n_kb_k=1\), \(a_k+b_k=a+b\) and each of \(a_k,b_k,m_k,n_k,t_k\) is equal to the corresponding \(a,b,m,n,t\) mod \(\lambda\).  It follows that \(M^{t_k}_{a_k,b_k}\) is diffeomorphic to \(M^t_{a,b}\) while \(s(M^{t_k}_{a_k,b_k},g^{t_k}_{a_k,b_k})\) is a polynomial in \(k\).  This completes the proof of \tref{cthm}.  
   
\begin{rem}

(a) One easily sees that \(W^7_{a,b}\) is diffeomorphic to only finitely many other \(W^7_{k,l}\), and no other homogeneous spaces.  By \cite{ST} Proposition 1.1 the space of \(G\) invariant metrics with nonnegative sectional curvature on a homogeneous space \(G/H\)  is connected.  Thus \(\mathfrak{M}_{\text{sec}\geq0}(W^7_{a,b})\) has infinitely many components by the corollary, but only finitely many of them contain homogeneous metrics.  Each of those in turn contains a positively curved metric, except in the case of \(W_{1,0}.\)  There are examples due to \cite{KS} where one has two components containing metrics with sec \(>0\).

(b) One sees from the diffeomorphism invariants that no two of the Eschenburg spaces \(F_{a,b}\) are diffeomorphic, so we cannot use this set of metrics to prove that any \(\mathfrak{M}_{\text{sec}>0}(F_{a,b})\) is not path connected.  

(c) We saw in the proof \lref{scal} that \(S^1\times SO(3)\) and \(U(2)\) respectively act by isometries on \(g^{t}_{a,b}\) and \(\bar{g}^{2t}_{a,b}\), and we suspect each is the full identity component of the isometry group. 

(d) The same argument as in \rref{hom} shows that \(M^{t}_{a,b}\) and \(\bar{M}^{2t}_{a,b}\) do not have the homotopy type of a 7-dimensional homogeneous space if \(|t(a+b)^2-ab|\) or \(|a^2-2tb^2|=2\) mod \(3\).

\end{rem}
 \providecommand{\bysame}{\leavevmode\hbox
to3em{\hrulefill}\thinspace}

\end{document}